\documentclass{amsart}

\usepackage[shortlabels]{enumitem}
\usepackage[mathscr]{euscript}
\usepackage{fouridx}
\renewcommand{\star}[1]{\ensuremath{{\fourIdx{*}{}{}{}{#1}}}}
\newcommand{\std}[1]{{\fourIdx{\circ}{}{}{}{#1}}}

\newcommand{\bfa}{\mathbf{a}}
\newcommand{\bfb}{\mathbf{b}}
\newcommand{\bfc}{\mathbf{c}}
\newcommand{\bfx}{\mathbf{x}}
\newcommand{\bfy}{\mathbf{y}}

\newcommand{\R}{\mathbb{R}}
\newcommand{\N}{\mathbb{N}}
\newcommand{\C}{\mathbb{C}}

\newcommand{\MFF}{\mathfrak{F}}

\newcommand{\SK}{\mathscr{K}}

\newtheorem{Thm}{Theorem}
\newtheorem*{Thm*}{Theorem}

\newtheorem{Prop}{Proposition}

\title[]{Nonstandard arguments for results about infinite systems of equations in infinitely many variables}

\author{David A. Ross}
\address{Department of Mathematics\\University of Hawaii at Manoa\\
Honolulu, HI 96822}
\email{ross@math.hawaii.edu}



\begin{document}

\begin{abstract}Short nonstandard proofs are given for some results about infinite systems of equations in infinitely many variables.
\end{abstract}

\maketitle

This note is motivated by Nathanson~\cite{nath24x}, which describes several results of the form:
\begin{quote}
{If some property P is true for every finite subset of a set of objects, then P is true simultaneously for all the objects.}
\end{quote}

For example, he examines a result of Abian and Eslami~\cite{abia-esla82} about certain sets $\MFF$ of equations in infinitely many variables, where it is shown that if every finite subset of $\MFF$ has a common solution of some form then $\MFF$ has a common solution of that form.

Results of this form are familiar to mathematical logicians as \emph{compactness} results.
In this note I give nonstandard proofs for some of the results in Nathanson~\cite{nath24x}, as well as others.  In nonstandard analysis the compactness principle is built into the construction of the nonstandard model, and so these results become completely transparent with the nonstandard machinery.  This also makes it possible to identify natural generalizations of the theorems; see, for example, Theorem~\ref{DAR:thm:epsilon} below.

Questions about infinite systems of equations in infinitely many variables go back to the early days of functional analysis; arguably the starting point for the whole field.  The Abian--Eslami result is a relative of a 1913 result of F. Riesz~\cite{ries13}.  Riesz's proof, reproduced in Banach's 1932 Theory of Linear Operations \cite{Banach32}\footnote{an extension of Banach's 1922 thesis} uses something today we would recognize as a special case of the Hahn--Banach Extension Theorem (Hahn~\cite{hahn27}, Banach~\cite{Banach29}), which should probably be attributed to Helly (\cite{helly12}; see Narici and Beckenstein~\cite{NariciBeckenstein} for the historical discussion).

It should also be noted that nonstandard analysis has proved a useful tool in understanding the Hahn--Banach Theorem.  In 1962 Wim Luxemburg~\cite{lux62} showed that the Boolean Prime Ideal Theorem  implies Hahn--Banach.  A few years later he refined the result \cite{lux69} to show that Hahn--Banach is equivalent to the statement that  every Boolean algebra has a $[0,1]$-valued measure.  Both these proofs used nonstandard analysis.  See David Pincus~\cite{Pincus}  for the proof that the Boolean Prime Ideal Theorem is strictly stronger than Hahn--Banach, as well as a comprehensive discussion of Luxemburg's results and other equivalents and inequivalents to the Hahn--Banach Theorem.

This paper can be viewed as a companion to the forthcoming Nathanson and Ross~\cite{NR}, which uses Tychonoff's Theorem to `standardize' some of these nonstandard proofs, and finds some further generalizations.

\section{Finite polynomials}

\begin{Thm}[Abian~\cite{Abian1972}]\label{thm:finitering}Let $R$ be a finite ring, $\MFF$ a set of (finite) polynomials over (possibly infinitely many) variables $\{x_i\}_{i\in I}$, and suppose every finite subset of $\MFF$ has a common root in $R$:
\[\forall p_1,\dots,p_n\in\MFF\ \exists \bfx=\{x_i\}{_i\in I}\subseteq R\textrm{ such that }p_k(\bfx)=0\quad (1\le k\le n)\]
Then there exists $\bfx\in R^I$ such that $p(\bfx)=0$ for every $p\in\MFF$.
\end{Thm}
\begin{proof}[Proof of Theorem~\ref{thm:finitering}]
By saturation there exists a $\bfx\in\star{\R_I}$ such that $\star{p}(\bfx)=0$ for every $p\in\MFF$.
Since $R$ is finite, $\star{R}=R$.  It follows that for $i\in I$ standard, $x_i\in R$, so for any standard $p$, $p(\bfx|_I)=\star{p}(\bfx)$.  In particular, $p(\bfx|_I)=0$ for every $p\in\MFF$.
\end{proof}

I note that this result is also an easy exercise in elementary model theory.  Let $\mathcal{L}=\langle \mathbf{0},\mathbf{1},+,\times,\mathbf{r}_i, \mathbf{c}_j\rangle_{i\le n, j\in I}$ be the first order language of rings, supplemented with constants $\mathbf{r}_i$ for every element of the ring $R$ and new constants $\mathbf{c}_j$.  Let $\MFF'$ be the set
\[ \{p(\overline{\mathbf{c}})=\mathbf{0} : p(\overline{x})\in MFF\} \]

Then $Th(R)\cup \MFF'$ is finitely consistent, so by the Compactness Theorem it has a model, which must be isomorphic to $R$.  Moreover, the interpretation $\bfx$ of $\overline{c}$ in this copy of $R$ is a solution to $\MFF$.

This result does not easily extend to infinite rings.  For example, Abian gives the following example:

\begin{gather*}
(x-1)-y_1^2\\ (x-2)-y_2^2\\ (x-3)-y_3^2\\ \vdots\\ (x-n)-y_n^2\\ \vdots
\end{gather*}

Any finite subset of these has a solution in $\R$, but the full set does not.  What makes the example work is that the finite satisfiability of the equations requires arbitrarily large solutions.  If we bound the solutions, then we do get a positive result.

\begin{Thm}\label{thm:bounded}Let $\MFF$ a set of (finite) real polynomials over (possibly infinitely many) variables $\{x_i\}_{i\in I}$, and suppose every finite subset of $\MFF$ has a common root in $\displaystyle[-M, M]^I$ for some $M$:
\[\forall p_1,\dots,p_n\in\MFF\ \exists \bfx=\{x_i\}_{i\in I}\subseteq [-M, M]\textrm{ such that }p_k(\bfx)=0\quad (1\le k\le n)\]
Then there exists $\bfx\in [-M, M]^I$ such that $p(\bfx)=0$ for every $p\in\MFF$.
\end{Thm}

\begin{proof}[Proof of Theorem~\ref{thm:bounded}]
By saturation there exists a $\bfx\in\star{\R_I}$ such that $\star{p}(\bfx)=0$ for every $p\in\MFF$.
For $i\in I$ standard, $|x_i|\le\Phi(i)$, so $y_i=\std{x_i}$ exists.  Put $\bfy=\{y_i\}_{i\in I}$.
Now, if $p\in\MFF$ is standard then its coefficients are all standard and its variables all have standard indices; it follows that $p(\bfy)=\std{\star{p}(\bfx)}=0$, proving the theorem.\end{proof}

\subsection*{Remarks}
The reader should compare this proof to the nonstandard proof of Tychonoff's Theorem for Hausdorff spaces.

With the obvious modification Theorem~\ref{thm:bounded} is true in $\C$ as well, but in fact Abian has showed that for $\C$ the bound is not necessary.

The proof works \emph{mutatis mutandis} if we choose a different $M_i$ for each $i$; that is, change the hypothesis/conclusion to $\bfx\in\prod_{i\in I}[-M_i, M_i]$.

In fact, this result has an obvious generalization:

\begin{Thm}\label{thm:common}Let $\MFF$ a set of functions over (possibly infinitely many) variables $\{x_i\}_{i\in I}$ with the property that each $f\in\MFF$ (a)~only depends on finitely many variables in $I$, and (b)~is continuous in each variable.  Let $\SK$ be a subset of $\R^I$ which is compact in the product topology. Suppose every finite subset of $\MFF$ has a common root in $\SK$:
\[\forall f_1,\dots,f_n\in\MFF\ \exists \bfx=\{x_i\}_{i\in I}\in\SK\textrm{ such that }f_k(\bfx)=0\quad (1\le k\le n)\]
Then there exists $\bfx\in\SK$ such that $p(\bfx)=0$ for every $p\in\MFF$.
\end{Thm}

\section{Functions of infinitely many variables}

What if the functions are functions of infinitely many variables? The functional analysts of the early 20th century understood that even for linear functions of infinitely many variables, constraints needed to be added.

For example, Helly~\cite{helly21} offered the following cautionary example in 1921:

\begin{align*}
1=&x_1+x_2+x_3+x_4+x_5+\cdots\\
1=&\phantom{x_1+}\ x_2+x_3+x_4+x_5+\cdots\\
1=&\phantom{x_1+x_2+}\ x_3+x_4+x_5+\cdots\\
&\phantom{x_1+x_2+}\vdots
\end{align*}

I now give a short nonstandard proof of the following theorem of Abian and Eslami~\cite{abia-esla82}:

\begin{Thm}[Abian and Eslami 1982]\label{thetheorem}
Let $M>0$, and $p,q>1$ be conjugate, ie $\frac{1}{p}+\frac{1}{q}=1$.  Let $I$ be an infinite set, $\bfa_i\in\ell^p$ and $b_i\in\R$ for all $i\in I$, and suppose that for every finite $I'\subseteq I$
the equations
\[\sum_{n=0}^\infty a_{in}x_n=b_i\qquad\qquad (i\in I')\]
have a simultaneous solution $\bfx\in\ell^q$ with $\|\bfx\|_q\le M$.
(Equivalently, for every finite $I'\subseteq I$ there exists $\bfx\in\ell^q$ with   $\|\bfx\|_q\le M$ and $\bfa_i\cdot\bfx=b_i$ for $i\in I'$.)
 Then there exists $\bfx\in\ell^q$ with   $\|\bfx\|_q\le M$ and $\bfa_i\cdot\bfx=b_i$ for all $i\in I$.

\end{Thm}

Adopt the notation $\bfa=\{a_n\}_{n\in\N}$ for a sequence of real numbers, $\bfa_i=\{a_{in}\}_{n\in\N}$ for a sequence of sequences of real numbers indexed by $i$, and $\bfa\cdot\bfb=\sum_{n\in\N}a_nb_n$.

If $N\in\N$ let $\bfa^N$ be the sequence:
\[a^N_n= \begin{cases}
        a_n & \text{if } n>N \\
        0 & \text{otherwise}
    \end{cases}
\]

First, a couple of simple facts about $\ell^q$ (and of course $\ell^p$).

\begin{Prop}\label{theprop}Let $\bfc\in\ell^q$ and $N\in\N$.  Then:
\begin{enumerate}[(i)]
\item $\bfc^N\in\ell^q$
\item $\lim_{N\to\infty}\|\bfc^N\|_q=0$
\item $|c_N|\le\|\bfc\|_q$
\end{enumerate}
\end{Prop}

These are obvious, I think.

\begin{proof}[Proof of Theorem~\ref{thetheorem}.]

By saturation\footnote{We need our model to be either polysaturated, or card$(I)^+$--saturated, or to use concurrence like nonstandard analysts did pre-saturation.} there is an $\bfx\in\star{\ell^q}$ such that $\|x\|_q\le M$ and for every (standard) $i\in{I}, \star{\bfa_i}\cdot\bfx=b_i$

Define a standard sequence $y_n$, $n\in\N$, by putting $y_n=\std{x_n}$; this exists by Proposition~\ref{theprop}(iii).

It remains to prove three things:
\begin{enumerate}
\item $\bfy\in\ell_q$
\item $\|\bfy\|_q\le M$
\item $\bfa_i\cdot\bfy=b_i$ for (standard) $i\in I$
\end{enumerate}

For any $N\in\N$, $\sum_{n\le N}|y_n|^q\approx\sum_{n\le N}|x_n|^q\le M^q$; letting $N\to\infty$ proves (1) and (2).

For (3), let $N$ be a standard positive integer, then
\begin{align*}
|b_i-\sum_{n\le N}a_{in}y_n|&\approx  |b_i-\sum_{n\le N}a_{in}x_n|  \quad \textrm{(since $N$ is finite and $y_n\approx x_n$)}\\
&=\sum_{n> N}a_{in}x_n \qquad (\textrm{since }\star{\bfa_i}\cdot\bfx=b_i)\\
&=|\star{\bfa_i^N}\cdot\bfx|\qquad (\textrm{by definition of }\bfa_i^N)\\
&\le \|\star{\bfa_i^N}\|_p\|\bfx\|_q\qquad (\textrm{by }\star{\textrm{H\"{o}lder}})\\
&\le \|\bfa_i^N\|_pM\\
&\to 0\textrm{ as }N\to \infty \quad \textrm{(by Proposition~\ref{theprop}(ii)}
\end{align*}
which proves (3), and the theorem.
\end{proof}

\section{Generalizing the result}

We now consider two possible generalizations: easing the bound $M$ (as discussed in the remarks after  Theorem~\ref{thm:bounded}), and replacing the exact solution with an approximation.  The following does both at once:

\begin{Thm}\label{DAR:thm:epsilon}Let $p,q>1$ be conjugate, ie $\frac{1}{p}+\frac{1}{q}=1$.  Let $M_n>0$ for every $n\in\N$.
Let $I$ be an infinite set, $\bfa_i\in\ell^p$ and $b_i\in\R$ for all $i\in I$.
Then the following are equivalent:

\begin{enumerate}
\item There exists an $\bfx\in\ell^q$ with $\bfa_i\cdot\bfx=b_i$ for all $i\in I$ and $|x_n|\le M_n$ for all $n\in\N$.

\item There is a positive sequence $\{e_N\}_{N\in\N}$ with $\lim\limits_{N\to\infty}e_N=0$ such that for
every finite $I'\subseteq I$ and every $\epsilon>0$, there is an $\bfx$ with (a)~$\|\bfx^N\|_q\le{e_N}$ for all $N$,
(b)~$|x_n|\le M_n$ for all $n\in\N$, and (c)~$|\bfa_i\cdot\bfx-b_i<\epsilon$ for $i\in I'$

\end{enumerate}

\end{Thm}
\begin{proof}\
$\mathbf{(1\Rightarrow2)}$ Given $\bfx$, let $e_N=\|\bfx^N\|_q$.  (2) is immediate.

$\mathbf{(2\Rightarrow1)}$  By saturation there is an $\bfx\in\star{\ell^q}$ such that (a)~$\|\bfx^N\|_q\le{e_N}$ for all standard $N$,
(b)~$|x_n|\le M_n$ for all $n\in\star\N$ (including standard $n$), and (c)~$|\star{\bfa_i}\cdot\bfx-b_i|\approx0$ for $i\in I$.

Put $y_n=\std{x_n}$ for standard $n$.  As in the proof of Theorem~\ref{thetheorem}, it remains to prove three things:
\begin{enumerate}
\item $\bfy\in\ell_q$
\item $|y_n|\le M_n$ for all $n\in\N$
\item $\bfa_i\cdot\bfy=b_i$ for (standard) $i\in I$
\end{enumerate}

(2) is immediate.  If $N<K$ are standard then

\[\sum_{n=N+1}^K|y_n|^q\approx \sum_{n=N+1}^K|x_n|^q\le \|\bfx^N\|_q^q\le{e_N^q}\]

Let $K\to\infty$, obtain $\|\bfy^N\|_q\le{e_N}$.  By the earlier observation, $\bfy\in\ell^q$, proving (1).

For (3), let $N$ be a standard positive integer, then
\begin{align*}
\left|b_i-\sum_{n\le N}a_{in}y_n\right|&\le\left|b_i-\star{\bfa_i}\cdot\bfx\right|+\left|\sum_{n\le N}a_{in}x_n-\sum_{n\le N}a_{in}y_n\right|+\left|\star{\bfa_i}\cdot\bfx-\sum_{n\le N}a_{in}x_n\right|\\
&\approx |\star{\bfa_i}\cdot\bfx^N|\qquad (\textrm{by definition of }\bfx^N)\\
&\le \|\star{\bfa_i}\|_p\|\bfx^N\|_q\qquad (\textrm{by }\star{\textrm{H\"{o}lder}})\\
&\le \|\bfa_i\|_pe_N\\
&\to 0\textrm{ as }N\to \infty
\end{align*}
where for the second line the first term is second term is infinitesimal since $x_n\approx y_n|$ for every $n\le N$ and the second term is infinitesimal by choice of $\bfx$.  This proves (3), and the theorem.
\end{proof}


\begin{thebibliography}{30}

\bibitem{Abian1972} Abian, A, \emph{On solvability of infnite systems of polynomial equations over a finite ring}, Simon Stevin 46 (1972/73) 33--37.


\bibitem{abia-esla82}
Abian, A. and Eslami, E., \emph{Solvability of infinite systems of infinite
  linear equations}, Bol. Soc. Mat. Mexicana (2) \textbf{27} (1982), no.~2,
  57--64.

\bibitem{Banach29} Banach, S., \emph{Sur les fonctionelles lin\'eaires}, Studia Math. 1 (1929)

\bibitem{Banach32}Banach, S., \emph{Th\'eorie des op\'erations lin\'eaires}, Chelsea, New York (1932).  Translated and reprinted as \emph{Theory of Linear Operators}, Dover Publications, Mineola, NY (1987).


\bibitem{hahn27}Hahn, H., \emph{\"Uber linearer Gleichungssysteme in linearer Raumen},
J. Reine Angew. Math. 157 (1927), 214-229.

\bibitem{helly12} Helly, E., \emph{\"Uber linearer Funktionaloperationen}, Sitzungsber. der math.
Naturwiss. Klasse der Akad. der Wiss. (Wien), 121 (1912) 265-297.
[

\bibitem{helly21} Helly, E., \emph{\"Uber Systeme linearer Gleichungen mit unendlich vielen
Unbekannten}, Monatsh. f\"ur Math. Phys. 31 (1921), 60-91.


\bibitem{lux62} Luxemburg, W. A. J., \emph{2 applications of the method of construction
by ultrapowers to analysis}, Bull. A.M.S. 68 (1952), 416--419.

\bibitem{lux69} Luxemburg, W. A. J., \emph{Reduced powers of the real number system
and equivalents of the Hahn-Banach extension theorem}, Int.
Sympos on the Application of Model Theory to Alg. Anal, and
Probability, Holt, Rinehart and Winston (1969)


\bibitem{NariciBeckenstein} Narici, L., and Beckenstein, E., \emph{The Hahn-Banach Theorem: the life and times}, Topology and its Applications  77 (1997) 193--211

\bibitem{nath24x}
Nathanson, M. B., \emph{Finitely many implies infinitely many} (2023);
arXiv: 2401.12887.

\bibitem{NR}
Nathanson, M. B.  and Ross, D. A., \emph{Multiplicative polynomial equations in infinitely many variables} (2024); arXiv:2405.01766



\bibitem{Pincus} Pincus, D., \emph{The Strength of the Hahn-Banach Theorem},  Victoria Symposium
on Nonstandard Analysis, University of Victoria 1972, A. Hurd and P. Loeb (eds.), Lecture Notes in Mathematics 369, Springer-Verlag, Berlin (1974) 203--248

\bibitem{ries13}
Riesz, F., \emph{Les syst{\` e}mes d'{\' e}quations lin{\'  e}aires {\` a}
une infinit{\' e} d'inconnues}, Gauthier-Villars, Paris,1913.

\end{thebibliography}
\end{document}